\documentclass[12pt, reqno]{amsart}

\author[S.~Cerreia-Vioglio]{Simone Cerreia-Vioglio}
\address{Universit\`a ``Luigi Bocconi''\\Department of Decision Sciences\\Milan, Italy}
\email{simone.cerreia@unibocconi.it}

\author[P.~Leonetti]{Paolo Leonetti}
\address{Universit\`a ``Luigi Bocconi''\\Department of Decision Sciences\\Milan, Italy}
\email{paolo.leonetti@unibocconi.it}

\author[F.~Maccheroni]{Fabio Maccheroni}
\address{Universit\`a ``Luigi Bocconi''\\Department of Decision Sciences\\Milan, Italy}
\email{fabio.maccheroni@unibocconi.it}


\title{A Characterization of the Vector Lattice of \\  Measurable Functions}

\usepackage{amsmath}
\usepackage{amssymb}
\usepackage{amsthm}
\usepackage[left=3.5cm, right=3.5cm, paperheight=11.8in]{geometry}
\usepackage{color}
\usepackage{hyperref}
\usepackage{enumerate}
\usepackage{bm}
\usepackage{comment}
\usepackage{fancyhdr}
\usepackage{nicefrac}
\usepackage{mathrsfs}
\usepackage[inline]{enumitem}

\newtheorem{thm}{Theorem}[section]
\newtheorem{cor}[thm]{Corollary}
\newtheorem{lem}[thm]{Lemma}

\theoremstyle{definition} 

\newtheorem{example}[thm]{Example}
\let\oldexample\example
\renewcommand{\example}{\oldexample\normalfont}
\newcommand{\aaz}{\mathrm{e}}

\theoremstyle{remark}
\newtheorem{claim}{\textsc{Claim}}
\newtheorem*{claim*}{\textsc{Claim}}

\AtBeginDocument{%
   \def\MR#1{}
}

\providecommand{\MR}[1]{}
\providecommand{\bysame}{\leavevmode\hbox to3em{\hrulefill}\thinspace}
\providecommand{\MR}{\relax\ifhmode\unskip\space\fi MR }

\providecommand{\href}[2]{#2}

\hypersetup{
    pdftitle={A Characterization of the Vector Lattice of Measurable Functions},
    pdfauthor={Paolo Leonetti},
    pdfmenubar=false,
    pdffitwindow=true,
    pdfstartview=FitH,
    colorlinks=true,
    linkcolor=blue,
    citecolor=green,
    urlcolor=cyan
}

\usepackage[italian,english]{babel}
\usepackage[utf8]{inputenc}

\begin{document}

\maketitle
\thispagestyle{empty}

\begin{abstract}
\noindent{} Given a probability measure space $(X,\Sigma,\mu)$, it is well known that the Riesz space $L^0(\mu)$ of equivalence classes of measurable functions $f: X \to \mathbf{R}$ is universally complete and the constant function $\bm{1}$ is a weak order unit. Moreover, the linear functional $L^\infty(\mu)\to \mathbf{R}$ defined by $f \mapsto \int f\,\mathrm{d}\mu$ is strictly positive and order continuous. Here we show, in particular, that the converse holds true, i.e., any universally complete Riesz space $E$ with a weak order unit $e>0$ which admits a strictly positive order continuous linear functional on the principal ideal generated by $e$ is lattice isomorphic onto $L^0(\mu)$, for some probability measure space $(X,\Sigma,\mu)$.
\end{abstract}

\section{Introduction}\label{sec:introduction}

A classical result of Kakutani 
\cite{MR0004095} 
states that every AL-space, that is, every Banach lattice with the norm additive on pairs of positive disjoint vectors, has to be a space $L^1(\mu)=L^1(X,\Sigma,\mu)$ of equivalence classes of $\mu$-integrable functions $f:X \to \mathbf{R}$, where 
$\mu: \Sigma \to [0,\infty]$ is a $\sigma$-additive measure. 
In addition, if there exists a weak order unit, then $\mu$ can be chosen finite. 
This is a 
characterization of the class of integrable functions by properties of the norm and order.

Relying on this result, Masterson \cite{MR0233177} proved a classification for the set of (equivalence classes of) real-valued measurable functions (see Section \ref{subsec:notation} for definitions): 
\begin{thm}\label{th:masterson}
Let $E$ be an Archimedean Riesz space. Then there exists an onto lattice isomorphism $E \to L^0(\mu)$, for some 
$\sigma$-finite measure space $(X,\Sigma,\mu)$, if and only if $E$ is 
universally complete, has the countable sup property, 
and the extended order continuous dual of $E$ is separating on $E$.
\end{thm}

Note that Theorem \ref{th:masterson} involves only order properties. Here, the extended order continuous dual of $E$, usually denoted by $\Gamma(E)$, is the set of equivalence classes of order continuous linear functionals 
defined on order dense ideals of $E$, where two functionals are identified whenever 
they agree on an order dense ideal of $E$, cf. \cite[\S1]{MR0212540}. 
It is well known that $\Gamma(E)$ is separating on $E$ if and only if there exists an order dense ideal $I$ of $E$ such that the order continuous dual of $I$ is separating on $I$, and in that case there exists an order dense ideal which admits a strictly positive order continuous linear functional, see \cite[Theorem 2.5]{MR0212540}. Other equivalent conditions are provided in \cite[Theorem 3.4]{MR1307578}; in particular, if $\Gamma(E)$ is separating on $E$, then there exists a measure space $(X,\Sigma,\mu)$ for which $E$ can be embedded order densely into $L^0(\mu)$.

Related results concerning representations of Archimedean Riesz spaces as spaces of measurable functions can be found, e.g., in Pinsker \cite{MR0021661}, Fremlin \cite{MR0216272}, and Labuda \cite{MR892050}, and are surveyed by Filter \cite[Section 3]{MR1307578}.

The aim of this work is to obtain a concrete characterization of the space of (equivalence classes of) measurable real-valued functions $L^0(X,\Sigma,\mu)$, where $\mu: \Sigma \to \mathbf{R}$ is a \emph{probability} measure, which is analogous to Theorem \ref{th:masterson}, and relies more on algebraic than on order properties (see also \cite[Chapter 36]{MR2459668}). 
Remarkably, this characterization avoids the use of the extended order continuous dual, thus providing an operational criterion to establish when a vector lattice is necessarily a space of random variables, and the proof of our result is self-contained. 

In the recent years there has been a lot of research in $L^0$-modules and their applications. See, for example the works of Cerreia-Vioglio et al. \cite{MR3535749, MR3554767, MR3870592}, Doldi and Frittelli \cite{MR4340168}, Filipovi\'{c} et al. \cite{MR2968040}, Frittelli and Maggis \cite{MR2780777, MR3176684}, and Hoffmann et al. \cite{MR3483745}. 
An abstract characterization of $L^0(\mu)$ extends the scope of these applications to modules that are not \textit{prima facie} on $L^0(\mu)$, such as the modules on algebras of stochastic processes that are sometimes used in mathematical finance (e.g., modules on  the algebras of predictable and progressively measurable processes, see Doob \cite{MR1814344}).

Another advantage of this paper is introducing the possibility of working with ``the scalars'' of $L^0$-modules from a purely algebraic/functional analytic perspective. Dispensing with the ---sometimes cumbersome--- techniques needed to consider zero measure sets, a.s. null functions, and the induced quotient spaces.

Dually, our result delivers a concrete representation for $f$-algebras of $L^0$ type considered in  
\cite{MR3535749, MR3554767}, which was the original motivation for this work (see Section \ref{sec:mainconjecture} below).

\subsection{Notation}\label{subsec:notation} We refer to 
\cite{MR2011364} 
for basic aspects of Riesz spaces. Let $E$ be a Riesz space. Then, we denote the positive cone of a Riesz subspace $F$ by $F^+:=\{x \in F: x\ge 0\}$. 
A net $(x_\alpha)_{\alpha \in A}$ with values in $E$ is said to be order convergent to $x \in E$ if there exists a net $(y_\alpha)_{\alpha \in A}$ with the same index set satisfying $y_\alpha \downarrow 0$ and $|x_\alpha-x|\le y_\alpha$ for all $\alpha \in A$. 
A non-empty subset $S\subseteq E$ is said to be solid if $|y|\le |x|$ implies $y \in S$ whenever $x \in S$. 
The principal ideal generated by a vector $x \in E$, that is, the smallest solid Riesz subspace containing $x$, 
is denoted by $E_x$. 
A vector $e>0$ is called a \emph{strong order unit} if the principal ideal generated by $e$, namely, 
$$
E_e=\{y \in E: |y|\le \lambda e \text{ for some }\lambda \in \mathbf{R}\},
$$
coincides with $E$. Instead, $e$ is said to be a \emph{weak order unit} if, for each $x \in E$, there exists a net $(x_\alpha)_{\alpha \in A}$ with values in $E_e$ which is order convergent to $x$. 

$E$ is said to be \emph{laterally complete} [respectively, \emph{laterally} $\sigma$\emph{-complete}] if the supremum of every disjoint subset [resp., sequence] of $E^+$ exists in $E$. If $E$ is also Dedekind complete, then we say that $E$ is \emph{universally complete}. 
$E$ has the \emph{countable sup property} if for every subset $S$ of $E$ whose supremum exists in $E$, there exists an at most countable subset of $S$ having the same supremum as $S$ in $E$. 
A (not necessarily Hausdorff) topology $\tau$ on a Riesz space $E$ is said to be \emph{locally solid} if $\tau$ has a base at zero consisiting of solid sets.

As usual, a probability measure space $(X,\Sigma,\mu)$ is a non-empty set $X$, together with a $\sigma$-algebra $\Sigma$ of subsets of $X$, and a $\sigma$-additive measure $\mu: \Sigma \to \mathbf{R}$ with $\mu(X)=1$.  
Moreover, $\bm{1}$ stands for the multiplicative unit of $L^0(\mu)$, whenever the underlying measure space is understood. Finally, given an integrable function $f \in L^1(\mu)$, we shorten $\int f \,\mathrm{d}\mu$ with $\mu(f)$.


\section{The characterization}\label{sec:mainconjecture}

We start with a preliminary observation, whose proof is given in Section \ref{sec:mainproof}.
\begin{lem}\label{lemmametricE}
Let $E$ be a Riesz space with weak order unit $e>0$ 
and let $\varphi: E_e \to \mathbf{R}$ be a strictly positive linear functional. Then 
\begin{equation}\label{eq:dvarphi}
d_\varphi: E \times E\to \mathbf{R}: (x,y) \mapsto \varphi(|x-y|\wedge e)
\end{equation}
is an invariant metric and the topology $\tau_\varphi$ generated by $d_\varphi$ is Hausdorff locally solid.
\end{lem}

Our main result follows.
\begin{thm}\label{mainconjecture}
Let $E$ be a Dedekind complete Riesz space with weak order unit $e>0$. Then the following are equivalent:
\begin{enumerate}[label={\rm (\roman{*})}]
\item\label{item:conj1} There exist a probability measure space $(X,\Sigma,\mu)$ and an onto lattice isomorphism  
$T: E \to L^0(\mu)$ such that $T(e)=\bm{1}$.
\item\label{item:conj2} There exists a strictly positive order continuous linear functional $\varphi: E_e \to \mathbf{R}$ such that the metric $d_\varphi$ is complete.
\item\label{item:conj3} There exists a strictly positive order continuous linear functional  
$\psi: E_e \to \mathbf{R}$ and $E$ is laterally complete.
\end{enumerate}
Moreover, in such case, $E_e$ is lattice isomorphic onto $L^\infty(\mu)$, the metrics $d_\varphi$ and $d_\psi$ are topologically equivalent, i.e., $\tau_\varphi=\tau_\psi$, and $E$ has the countable sup property.
\end{thm}

The implication \ref{item:conj2} $\implies$ \ref{item:conj1} is related to \cite[Theorem 6.4]{MR2122234}, which characterizes norm dense ideals of $L^1(\mu)$. To the best of our knowledge, the equivalence \ref{item:conj1} $\Longleftrightarrow$ \ref{item:conj3} is completely new.

As an immediate consequence of Theorem \ref{mainconjecture}, we obtain a result in the same spirit of Theorem \ref{th:masterson}. Indeed, recalling that $L^0(\mu)$ is universally complete \cite[Theorem 7.73]{MR2011364} and has a weak order unit $\bm{1}$, it follows that (we omit details): 
\begin{cor}\label{cor:mast1}
Let $E$ be an Archimedean Riesz space. Then $E$ is lattice isomorphic onto $L^0(\mu)$, for some probability measure space $(X,\Sigma,\mu)$, if and only if $E$ is universally complete 
\textup(hence, with weak order unit $e>0$\textup) and admits a strictly positive order continuous linear functional on $E_e$. 
\end{cor}

Finally, we obtain a charaterization of $f$-algebras of $L^0$ type, cf. \cite[Definition 6]{MR3554767}. 
In this regard, we recall that an $f$-algebra is a Riesz algebra $E$ 
for which $(a\cdot c) \wedge b=(c\cdot a) \wedge b=0$ for all $a,b,c \ge 0$ such that $a\wedge b=0$. If, in addition, $E$ is Dedekind complete and admits a non-zero multiplicative unit $e$, then it is said to be a Stonean algebra, cf. \cite[Definition 2]{MR3535749}. In such case, the following facts are well known and readily provable:
\begin{enumerate*}[label={\rm (\roman{*})}]
\item The multiplication is commutative, i.e., $a\cdot b=b\cdot a$ for all $a,b \in E$,
\item $x^2:=x \cdot x \ge 0$ for all $x \in E$; in particular, $e>0$, and 
\item $e$ is a weak order unit.
\end{enumerate*} 

Accordingly, a Stonean algebra $E$ is said to be $f$\emph{-algebra of} $L^0$ \emph{type} whenever the principal ideal $E_e$ is an Arens algebra, i.e., a real commutative Banach algebra such that $\|e\|=1$ and $\|a\|^2 \le \|a^2+b^2\|$ for all $a,b \in E_e$, and there exists a strictly positive order continuous linear functional $\varphi$ on $E_e$ such that the metric $d_\varphi$ defined in \eqref{eq:dvarphi} is complete. 
As an application, Theorem \ref{mainconjecture} implies that $f$-algebras of $L^0$ type are (equivalence classes of) spaces of random variables.
\begin{cor}\label{cor:falgebra}
Let $E$ be an Archimedean $f$-algebra with non-zero multiplicative unit. Then $E$ is an $f$-algebra of $L^0$ type if and only if $E$ is lattice and algebra isomorphic onto $L^0(\mu)$, for some probability measure space $(X,\Sigma,\mu)$. 
\end{cor}

Finally, it is worth noting that the topological equivalence of $d_\varphi$ and $d_\psi$ at the end of Theorem \ref{mainconjecture} cannot be strengthened to strongly equivalence, as it is shown in the following example.
\begin{example}\label{example:notstronglyequivalent}
Let $\mu$ be the function $\mathcal{P}(\mathbf{N}) \to \mathbf{R}: X \mapsto \sum_{x \in X} 2^{-x}$, where $\mathbf{N}$ is the set of positive integers and $\mathcal{P}(\mathbf{N})$ its powerset. 
Then $(\mathbf{N},\mathcal{P}(\mathbf{N}),\mu)$ is a probability measure space, $L^0(\mu)$ is the space of real-valued sequences (indexed by $\mathbf{N}$), and $L^\infty(\mu)$ is the ideal generated by $e=(1,1,\ldots)$, i.e., the subspace of bounded sequences $\ell^\infty$. 
Accordingly, define the strictly positive order continuous linear functionals $\varphi: \ell^\infty\to \mathbf{R}$ and $\psi: \ell^\infty \to \mathbf{R}$ mapping each $x=(x_1,x_2,\ldots)$ into $\sum_{n\ge 1}x_n2^{-n}$ and $\sum_{n\ge 1}x_n3^{-n}$, respectively. 

With this, let us suppose for the sake of contradiction that there exists a positive constant $c$ such that $d_\varphi(x,y) \le cd_\psi(x,y)$ for all $x,y \in L^0(\mu)$. Moreover, for each $n \in \mathbf{N}$, define 
$e_n=(0,\ldots,0,1,1,\ldots)$, where $0$ is repeated exactly $n$ times. Then, it would follow 
$$
\textstyle \sum_{k\ge n}2^{-k}=\varphi(e_n)=d_\varphi(e_n,0) \le cd_\psi(e_n,0)=c\psi(e_n)= c\sum_{k\ge n}3^{-k},
$$
which is false whenever $n$ is sufficiently large.
\end{example}

Proofs of Theorem \ref{mainconjecture} and 
Corollary \ref{cor:falgebra} 
follow in Section \ref{sec:proofs}.

\section{Preliminaries}\label{sec:mainproof}

We start with the proof of Lemma \ref{lemmametricE}.
\begin{proof}[Proof of Lemma \ref{lemmametricE}]
Note that $d_\varphi$ is well defined since $E_e$ is solid and $0\le |x-y|\wedge e \le e \in E_e$ for all $x,y \in E$. Since $e$ is a weak order unit, $|x-y|\wedge e=0$ if and only if $x=y$. Then, the strict positivity of $\varphi$ implies that $d_\varphi(x,y)=d_\varphi(y,x) \ge 0$ for all $x,y \in E$, with equality if and only if $x=y$. Finally, for each $x,y,z \in E$, we have $|x-z| \le |x-y|+|y-z|$, so that, thanks to 
\cite[Theorem 1.7.(4)]{MR2011364}, 
$$
|x-z|\wedge e \le (|x-y|+|y-z|)\wedge e \le |x-y|\wedge e+|y-z|\wedge e.
$$
Since $\varphi$ is a positive operator, we obtain $d_\varphi(x,z) \le d_\varphi(x,y)+d_\varphi(y,z)$. Clearly, $d_\varphi$ is invariant and $(E,\tau_\varphi)$ is Hausdorff. 

Finally, the local solidness follows by the fact each open ball $B$ centered in $0$ and with radius $r>0$ is solid. Indeed, given $x,y \in E$ with $|x|\le |y|$ and $y \in B$, then by the positivity of $\varphi$ we get 
$\varphi(|x|\wedge e) \le \varphi(|y|\wedge e)$, that is, $x \in B$. 
\end{proof}

The following result is classical, hence we omit its proof.
\begin{lem}\label{ordercontinuity}
Let $E,F$ be Riesz spaces and let $T:E\to F$ be an onto lattice isomorphism. Then $T$ is order continuous.
\end{lem}

Finally, we will use the following characterization of $L^\infty(\mu)$; cf. also Abramovich, Aliprantis, and Zame \cite[Corollary 2.2]{MR1331491}.

\begin{lem}\label{thm:AAZ} 
Let $E$ be a Dedekind complete 
Riesz space with strong order unit $e>0$ 
which admits a strictly positive order continuous linear functional $\varphi$. 
Then there exist a probability measure space $(X,\Sigma,\mu)$ and an onto lattice isomorphism 
$T: E\to L^\infty(\mu)$ such that $T(e)=\bm{1}$ and $\varphi(x)=\mu(T(x))$ for all $x \in E^+$.
\end{lem}
\begin{proof}
Since $\varphi$ is strictly positive, then $\varphi(e)>0$. Hence, dividing by $\varphi(e)$, we can suppose without loss of generality that $\varphi(e)=1$. It follows that
$$
\|\cdot\|:E\to \mathbf{R}: x\mapsto \varphi(|x|)
$$
is an order continuous L-norm. Let $\widehat{E}$ be the topological completion of $E$. Then, $\widehat{E}$ is an AL-space and, according to \cite[Footnote 6]{MR1331491}, $E$ is an (order dense) ideal of $\widehat{E}$. It follows by Kakutani's representation theorem \cite[Theorem 7]{MR0004095} that there exists an onto lattice and isometric $\widehat{T}: \widehat{E} \to L^1(\mu)$, for some  probability measure space $(X,\Sigma,\mu)$, such that $\widehat{T}(e)=1$. In particular,
$$
\varphi(x)=\mu(\widehat{T}(x))
$$
for all $x \in E^+$. In addition, since $e$ is unit of $E$ and $E$ is an ideal of $\widehat{E}$, then $E=E_e=\widehat{E}_e$. The claim follows by letting $T$ equal to the restriction of $\widehat{T}$ from $E$ to its direct image.
\end{proof}

\section{Proof of the Main Result}\label{sec:proofs}

\begin{proof}[Proof of Theorem \ref{mainconjecture}] 
We are going to show the following chain of equivalences: 
$$
\ref{item:conj1} \implies \ref{item:conj2} \implies \ref{item:conj3} \implies \ref{item:conj2} \implies \ref{item:conj1}.
$$

\bigskip

\ref{item:conj1}$\implies$\ref{item:conj2}. Let us assume that there exist a probability measure space $(X,\Sigma,\mu)$ and an onto lattice isomorphism $T: E \to L^0(\mu)$ such that $T(e)=\bm{1}$. In particular, $T$ is a positive operator. It follows that $T([-\lambda e,\lambda e])=[-\lambda T(e),\lambda T(e)]$, hence
\begin{equation}\label{eq:linfty}
\textstyle T(E_e)=T\left(\bigcup_{\lambda > 0}[-\lambda e, \lambda e]\right)=\bigcup_{\lambda > 0}[-\lambda \bm{1},\lambda \bm{1}]=L^\infty(\mu).
\end{equation}
Therefore, the restriction of $T$ on $E_e$, hereafter denoted by $T_e$, is a lattice isomorphism onto $L^\infty(\mu)$. Note that, thanks to Lemma \ref{ordercontinuity}, $T_e$ is order continuous.

At this point, define the linear functional  
$$
\varphi: E_e \to \mathbf{R}: x\mapsto \mu(T(x)).
$$
It is routine to check that $\varphi$ is strictly positive. Moreover, $\varphi$ is order continuous. To this aim, since $\varphi$ is a positive operator, it is enough to show that $\varphi(x_\alpha) \downarrow 0$ for every net $(x_\alpha) \downarrow 0$ in $E_e$. 
Since $\mathbf{R}$ is an Archimedean Riesz space with the countable sup property and $\varphi: E_e \to \mathbf{R}$ is strictly positive, it follows by \cite[Theorem 1.45]{MR2011364} that $E_e$ has the countable sup property as well. In particular, it is enough to show that $\varphi(x_n) \downarrow 0$ for every sequence $(x_n) \downarrow 0$ in $E_e$.  Since $T_e$ is order continuous, $T_e(x_n) \downarrow 0$ in $L^\infty(\mu)$. Finally $\varphi(x_n)=\mu(T_e(x_n)) \downarrow 0$ by Lebesgue's dominated convergence theorem.

Finally, we need to prove that the metric space $(E,d_\varphi)$ is (topologically) complete. Let $d$ be the metric of convergence in measure on $L^0(\mu)$, that is, 
$$
d: L^0(\mu)\times L^0(\mu)\to \mathbf{R}: (f,g) \mapsto \mu(|f-g|\wedge \bm{1}).
$$
Hence, for all $x,y \in E$, we obtain
\begin{equation}\label{eq:invariance}
\begin{split}
d_\varphi(x,y)&=\varphi(|x-y|\wedge e)=\mu(T(|x-y|\wedge e))=\mu(T(|x-y|)\wedge T(e))\\
&=\mu(|T(x-y)|\wedge \bm{1})=\mu(|T(x)-T(y)|\wedge \bm{1})=d(T(x),T(y)).
\end{split}
\end{equation}
Then, fix a Cauchy sequence $(x_n)$ of vectors in $E$, i.e., for each $\varepsilon>0$ there exists $n_0=n_0(\varepsilon)$ such that $d_\varphi(x_n,x_m) \le \varepsilon$ whenever $n,m \ge n_0$. It follows from \eqref{eq:invariance} that $(T(x_n))$ is a Cauchy sequence in $(L^0(\mu),d)$. Since the metric space $(L^0(\mu),d)$ is complete, there exists $f \in L^0(\mu)$ such that $d(T(x_n),f)\to 0$ as $n\to +\infty$. Moreover, $T$ is a bijection, hence there exists $x \in E$ such that $T(x)=f$. Therefore, thanks to \eqref{eq:invariance}, we obtain $d_\varphi(x_n,x)\to 0$ as $n\to +\infty$.

\bigskip

\ref{item:conj2}$\implies$\ref{item:conj3}. Suppose that there exists a strictly positive order continuous linear functional $\varphi:E_e\to \mathbf{R}$ for which the metric space $(E,d_\varphi)$ is complete, and set $\varphi=\psi$. 

Since $e$ is a weak order unit and $E$ is Dedekind complete, it follows by \cite[Theorem 7.39]{MR2011364} that it is enough to show that $E$ is laterally $\sigma$-complete. 
To this aim, let $(x_n)$ be a sequence of disjoint vectors in $E^+$ and define the sequences $(y_n)$ by $y_n:=x_n\wedge e$ for each $n\ge 1$. Note that $(y_n)$ is a disjoint sequence of vectors in the order interval $[0,e]$. Moreover, for each positive integer $n$, define
$$
a_n:=x_1+\cdots+x_n \,\,\,\text{ and }\,\,\,b_n:=y_1+\cdots+y_n.
$$
Since $E$ is Dedekind complete and $b_n=y_1 \vee \cdots \vee y_n \le e $ for each $n \ge 1$, then the supremum of the sequence $(b_n)$ exists in $[0,e]$, and we denote it by $b$. Hence $0\le b-b_n \downarrow 0$, which implies $0\le (b-b_n)\wedge e \downarrow 0$. Since $\varphi$ is order continuous, then 
$$
\lim_{n\to \infty}d_\varphi(b_n, b)=0.
$$
In particular, $(b_n)$ is a Cauchy sequence in $(E,d_\varphi)$. In addition, for all positive integers $n,m$ with $n>m$, it holds
\begin{displaymath}
\begin{split}
d_\varphi(a_n,a_m)&=\varphi((a_n-a_m)\wedge e)=\varphi((x_{m+1}\vee\cdots\vee x_n)\wedge e)=\varphi(y_{m+1}\vee\cdots\vee y_n) \\
&=\varphi((y_{m+1}+\cdots+y_n)\wedge e)=\varphi((b_n-b_m)\wedge e)=d_\varphi(b_n,b_m).
\end{split}
\end{displaymath}
It follows that also $(a_n)$ is a Cauchy sequence in $(E,d_\varphi)$. Since $(E,d_\varphi)$ is complete by hypothesis, there exists $a \in E$ such that 
\begin{equation}\label{eq:conv}
\lim_{n\to \infty}d_\varphi(a_n, a)=0.
\end{equation}

Thanks to Lemma \ref{lemmametricE}, $(E,\tau_\varphi)$ is a locally solid Hausdorff Riesz space. Therefore, according to \cite[Theorem 2.21.(c)]{MR2011364} and \eqref{eq:conv}, it follows that $x_1\vee \cdots \vee x_n=a_n \uparrow a$. By the previous argument, this implies that $E$ is laterally complete. 

\bigskip

\ref{item:conj3} $\implies$ \ref{item:conj2}. 
Set $\varphi=\psi$ and note that $\tau:=\tau_\psi$ is a Fatou topology on $E$, i.e., it has a neighborhood base at $0$ consisting of solid and order closed sets. 
Let $(\widehat{E},\hat{\tau})$ be the topological completion of $(E,\tau)$. 

According to a classical result of Nakano, see e.g. \cite[Theorem 4.28]{MR2011364}, since $(E,\tau)$ is a Dedekind complete locally solid Riesz space with the Fatou property, then the order intervals of $E$ are $\tau$-complete. Fix $x \in E$ and $\hat{x} \in \widehat{E}$ such that $0\le \hat{x} \le x$ in $\widehat{E}$ and let $(y_\alpha)$ be a net of positive vectors in $E$ such that $y_\alpha \overset{\tau}{\to} \hat{x}$. This implies that $x_\alpha \overset{\tau}{\to} \hat{x}$, where $x_\alpha:=y_\alpha \wedge x$ for each index $\alpha$. Since $x_\alpha \in [0,x]$ for each $\alpha$ and the order intervals are $\tau$-complete, then $\hat{x} \in E$. Hence $E$ is an ideal of $\widehat{E}$. (An alternative proof of this fact can be found also in \cite[Theorem 2.2]{MR0350372}.)

Moreover, given $0\le \hat{x} \in \widehat{E}$ and a net $(x_\alpha)$ of positive vectors in $E$ such that $x_\alpha \overset{\tau}{\to} \hat{x}$, then $x_\alpha \wedge \hat{x}$ belongs to $E^+$ (since $E$ is an ideal). Hence, considering the finite suprema of the net $(x_\alpha \wedge \hat{x})$, we obtain a net $(y_\beta)$ of vectors in $E^+$ such that $y_\beta \overset{\tau}{\to} \hat{x}$ and $y_\beta \uparrow \hat{x}$. This means that $E$ is an order dense ideal of $\widehat{E}$. 

Therefore, since $E$ is a universally complete order dense Riesz subspace of the Archimedean Riesz space $\widehat{E}$, then $E=\widehat{E}$ by the uniqueness of the universal completion, see e.g. \cite[Theorem 7.15.(ii)]{MR2011364}.

\bigskip

\ref{item:conj2} $\implies$ \ref{item:conj1}. Suppose that an increasing net $(x_\alpha)_{\alpha \in A}$ of positive vectors in $E_e$ is upper bounded by some $y \in E_e$. Then $x:=\sup\{x_\alpha: \alpha \in A\}$ exists in $E$ and belongs to the order interval $[0,y]$. Since $E_e$ is solid, then $x \in E_e$. Therefore, thanks to \cite[Lemma 1.39]{MR2011364}, $E_e$ is a Dedekind complete Riesz subspace with strong order unit $e>0$. It follows by Lemma \ref{thm:AAZ} that there exist a  
probability measure space $(X,\Sigma,\mu)$ and an onto lattice isomorphism 
$$
T_{\aaz}: E_e\to L^\infty(\mu)
$$
such that $T_{\aaz}(e)=\bm{1}$ and $\varphi(x)=\mu(T_{\aaz}(x))$ for all $0\le x \in E_e$. 
Then, for all $x,y \in E_e$ we obtain
\begin{equation}\label{eq:isometry}
\begin{split}
d(T_{\aaz}(x),T_{\aaz}(y))&=\mu(|T_{\aaz}(x)-T_{\aaz}(y)|\wedge \bm{1}) \\
&=\mu( T_{\aaz}(|x-y|)\wedge \bm{1} )=\mu( T_{\aaz}(|x-y| \wedge e) )\\
&=\varphi(|x-y| \wedge e)= d_\varphi(x,y).
\end{split}
\end{equation}

\begin{claim}\label{claim:closure}
$E$ is the topological closure of $E_e$ in $(E,\tau_\varphi)$.
\end{claim}
\begin{proof}
Given $x \in E^+$, then $(x_n) \uparrow x$, where $x_n:=x\wedge ne$, by the fact that $e$ is a weak order unit. This implies that $(|x-x_n|\wedge e) \downarrow 0$. Since $\varphi$ is order continuous, then
$$
d_\varphi(x_n,x)=\varphi(|x-x_n|\wedge e) \downarrow 0,
$$
i.e., $x_n \to x$ in $(E,\tau_\varphi)$. 
The claim follows by the fact that $x=x^+-x^-$ for each $x \in E$ and the topological limits are linear.
\end{proof}

\begin{claim}
There exists a positive operator $T:E \to L^0(\mu)$ extending $T_{\aaz}$ for which \eqref{eq:isometry} holds for all $x,y \in E$. 
\end{claim}
\begin{proof}
Define the operator $T:E\to L^0(\mu)$ as the unique extension of
\begin{equation}\label{eq:defT}
E^+\to L^0(\mu): x\mapsto \lim_{n\to \infty}T_{\aaz}(x_n),
\end{equation}
where $(x_n)_{n \ge 1}$ is any sequence in $E_e^+$ such that $x_n \to x$ in $(E,\tau_\varphi)$. The limit in \eqref{eq:defT} is understood to be in $(L^0(\mu),d)$. 

At first, we show that $T$ is well defined. To prove the existence of the limit, fix a sequence $(x_n)$ of vectors in $E_e$ such that $x_n \to x$ (note that such sequence exists by Claim \ref{claim:closure}). Then $(x_n)$ is a Cauchy sequence. It follows by \eqref{eq:isometry} that $(T_{\aaz}(x_n))$ is a Cauchy sequence in $(L^0(\mu),d)$. Then, by the completeness of the latter space, there exists (a unique) $f \in L^0(\mu)$ such that $\lim_{n\to \infty}T_{\aaz}(x_n)=f$.

Then, we show that the limit in \eqref{eq:defT} is independent from the choice of the sequence $(x_n)$. 
Indeed, let us suppose that $(x_n^\prime)$ is another sequence of vectors such that $x_n^\prime \to x$ in $(E,\tau_\varphi)$. This implies that $x_n-x_n^\prime \to 0$, i.e., 
$$
\lim_{n\to \infty}\varphi(|x_n-x_n^\prime|\wedge e)=0.
$$
Since $x_n-x_n^\prime \in E_e$ for each $n$ and $E_e$ is Dedekind complete, there exists $\ell \in E_e$ such that $\ell=\inf\{|x_n-x_n^\prime|:n\ge 1\}$. In particular, there exists a real $\lambda>0$ such that $\ell\le \lambda e$. Clearly, $\ell \ge 0$ and, by the strict positivity of $\varphi$, it follows that $\varphi(|x_n-x_n^\prime|\wedge e) \ge \varphi(\ell \wedge e)$ for all $n$, proving that $\ell \wedge e=0$. Hence $\ell=\ell \wedge \lambda e=0$. By the same argument, it is easy to see that there does not exist any $y>0$ in $E_e$ such that $|x_n-x_n^\prime|\ge y$ for infinitely many $n$. In particular, choosing $y=\nicefrac{1}{k}\,e$, we obtain that $x_n-x_n^\prime$ belongs to the order interval $[-\nicefrac{1}{k}\, e, \nicefrac{1}{k}\, e]$ whenever $n$ is sufficiently large. This implies that $x_n-x_n^\prime$ converges to $0$ with respect to the order, i.e.,
\begin{equation}\label{eq:Twelldef}
x_n-x_n^\prime \overset{o}{\to}0.
\end{equation}
Since $T_{\aaz}$ is a lattice isomorphism onto $L^\infty(\mu)$, then it is also order continuous, thanks to Lemma \ref{ordercontinuity}. Hence 
$
T_{\aaz}(x_n-x_n^\prime) \overset{o}{\to}0
$ 
in $L^\infty(\mu)$, which is equivalent to
$$
\lim_{n\to \infty}T_{\aaz}(x_n-x_n^\prime)(\omega)=0
$$
for each $\omega \in X$. Since it is well known that puntual convergence implies convergence in measure, then
$$
\lim_{n\to \infty} d(T_{\aaz}(x_n-x_n^\prime),0)=0,
$$
which is what we wanted to show. 

In addition, it is routine to check that $T$ is a positive operator.

Finally, for each $x,y \in E$, there exist by Claim \ref{claim:closure} two sequences of vectors $(x_n)$ and $(y_n)$ in $E_e$ such that $x_n \to x$ and $y_n \to y$ in $(E,\tau_\varphi)$. Thanks to Claim \ref{claim:closure} and \eqref{eq:isometry}, we get
\begin{displaymath}
\begin{split}
d_\varphi(x,y)=\lim_{n\to \infty}d_\varphi(x_n,y_n)&=\lim_{n\to \infty}d(T_{\aaz}(x_n),T_{\aaz}(y_n))\\
&=d(\lim_{n\to \infty}T_{\aaz}(x_n),\lim_{n\to \infty}T_{\aaz}(y_n))=d(T(x),T(y))
\end{split}
\end{displaymath}
for all $x,y \in E^+$, hence also for all $x,y \in E$. 
\end{proof} 

\begin{claim}
$T$ is an onto lattice isomorphism.
\end{claim}
\begin{proof}
Fix $0\le f \in L^0(\mu)$. Since the constant function $\bm{1}$ is a weak order unit, then $f_n \uparrow f$, where $f_n:=f\wedge n\bm{1}$ for each positive integer $n$. Again, since puntual convergence implies convergence in measure, we get $f_n \to f$ in $(L^0(\mu),d)$. In particular, $(f_n)$ is a Cauchy sequence. 

At this point, define $x_n:=T_{\aaz}^{-1}(f_n)$ for each $n$. Note that $(x_n)$ is a sequence of positive vectors in $E_e$ and, thanks to \eqref{eq:isometry}, is a Cauchy sequence in $(E,\tau_\varphi)$. Since the metric $d_\varphi$ is complete by hypothesis, there exists $x \in E^+$ for which $x_n \to x$. According to \eqref{eq:defT}, we conclude that
$$
T(x)=\lim_{n\to \infty}T_{\aaz}(x_n)=\lim_{n\to \infty} f_n=f,
$$
i.e., $f \in T(E)$. Then, by the arbitrariness of $f$, $T$ is onto, i.e., $T(E)=L^0(\mu)$.

To sum up, $T:E \to L^0(\mu)$ is a one-to-one and onto linear operator such that $T$ and $T^{-1}$ are both positive operators. Therefore, thanks to \cite[Exercise 16]{MR2011364}, $T$ is an onto lattice isomorphism.
\end{proof}

\bigskip

At this point, note that, if one of the equivalent conditions \ref{item:conj1}-\ref{item:conj3} hold, then $E_e$ is lattice isomorphic onto $L^\infty(\mu)$, thanks to \eqref{eq:linfty}. 

Also, the metrics $d_\varphi$ and $d_\psi$ are topologically equivalent: indeed, a laterally complete Riesz space admits at most one Hausdorff Fatou topology, which must be necessarily a Lebesgue topology (i.e., $x_\alpha \overset{\tau}{\to}0$ whenever $x_\alpha \downarrow 0$), see e.g. \cite[Theorem 7.53]{MR2011364}. 

Finally, suppose that $0\le x_\alpha \uparrow x$ in $E$, hence by the Lebesgue property $x-x_\alpha \overset{\tau}{\to}{0}$, i.e., $\varphi((x-x_\alpha) \wedge e) \to 0$. Then, there exists a subsequence $(x_{\alpha_n})$ of the net $(x_\alpha)$ such that $\varphi((x-x_{\alpha_n}) \wedge e) \to 0$ as $n\to \infty$, i.e., $x-x_{\alpha_n} \overset{\tau}{\to}{0}$. Considering that $x-x_{\alpha_n}$ is a decreasing sequence, we conclude that $x-x_{\alpha_n} \downarrow 0$ by 
\cite[Theorem 2.21.(c)]{MR2011364}, that is, $x_{\alpha_n}\uparrow x$. This means that $E$ has the countable sup property.
\end{proof}

Let us conclude with the proof of the last corollary.
\begin{proof}[Proof of Corollary \ref{cor:falgebra}]
If $E$ is lattice and algebra isomorphic onto $L^0(\mu)$, for some probability measure space $(X,\Sigma,\mu)$, then it is easy to check that $E$ is an $f$-algebra of $L^0$ type (we omit details).

Conversely, let us suppose that $E$ is an $f$-algebra of $L^0$ type. Then, in particular, $E$ is a Dedekind complete Riesz space with weak order unit $e>0$ and admits a strictly positive order continuous linear functional $\varphi: E_e \to \mathbf{R}$ such that the metric $d_\varphi$ defined in \eqref{eq:dvarphi} is complete. It follows by Theorem \ref{mainconjecture} that there exists a lattice isomorphism $T:E\to L^0(\mu)$, for some probability measure space $(X,\Sigma,\mu)$. 

Then, we have to prove that $T$ is also an algebra isomorphism. Note that the multiplication $\cdot$ defined by 
$$
x\cdot y:=T^{-1}(T(x)T(y))
$$
for all $x,y \in E$ makes $E$ an Archimedean $f$-algebra with multiplicative unit $e>0$. The claim follows by the fact that there exists at most one algebra multiplication on an Archimedean Riesz space $L$ that makes $L$ an Archimedean $f$-algebra with given unit, see e.g. \cite[Theorem 2.58]{MR2262133}.
\end{proof}

\subsection{Acknowledgements.} 
The authors are grateful to PRIN 2017 (grant \\  2017CY2NCA) for financial support, and to an anonymous reviewer for constructive comments which improved the overall presentation.


%
\bibliographystyle{amsplain}

\providecommand{\href}[2]{#2}

\end{document}